\documentclass[12pt]{article}

\usepackage[utf8]{inputenc}
\usepackage{amsmath, amssymb}
\usepackage{graphicx}
\usepackage{cite}
\usepackage{geometry}
\usepackage{setspace}
\usepackage{hyperref}
\usepackage{amsthm}
\usepackage{natbib}
\usepackage{xcolor}
\usepackage{float}
\usepackage{tikz}
\usepackage{bm}
\usepackage[normalem]{ulem}

\newtheorem{assumption}{Assumption}
\newtheorem{theorem}{Theorem}

\newtheorem{remark}{Remark}

\newtheorem{proposition}{Proposition}

\newcommand{\RNum}[1]{\rm \uppercase\expandafter{\romannumeral #1\relax}}
\DeclareMathOperator{\diag}{diag}

\title{Two energy methods for distributed port-Hamiltonian systems and their application to stability analysis}
\author{Hannes Gernandt, Marco Roschkowski \\
\normalsize University of Wuppertal \\
\normalsize \texttt{\{gernandt, roschkowski\}@uni-wuppertal.de}}

\date{\today}

\begin{document}

\maketitle
\begin{abstract}
We develop two local energy methods for distributed parameter port-Hamiltonian (pH) systems on one-dimensional spatial domains. The methods are applied to derive a characterization of exponential stability directly in terms of the energy passing through the boundary over a given time horizon. The resulting condition is verified  for a network of vibrating strings where existing sufficient conditions cannot be applied. Moreover, we use a~local energy method to study the short-time behavior of pH systems with boundary damping which was recently studied in the context of hypocoercivity.
\end{abstract}


\section{Introduction}

Port-Hamiltonian (pH) systems extend the Hamiltonian formalism by dissipation and interactions with an environment via ports, see \cite{SchJ14, ortega2008control, ramirez2022overview}. The pH modeling approach has been extended to infinite-dimensional systems which are often called distributed parameter pH systems~\cite{jacob2012linear, RasCSS20, augner2016stabilisation} and leads to systems in the Hilbert space~$L^2([a,b],\mathbb{R}^n)$  of the form 
\begin{align}
\nonumber
    \tfrac{\rm d}{{\rm d}t} x(t, \zeta) &=\left(\bm{P}_1\tfrac{\partial}{\partial \zeta} + \bm{G}_0\right)(\bm{H}(\zeta)x(t, \zeta))\\ x(0,\zeta)&=x_0(\zeta),\quad (t,\zeta)\in[0,\infty)\times[a,b],\label{pH}\\
 0 &=      \tilde{\bm{W}}_B \begin{pmatrix}
        (\bm{H}x)(b) \\
        (\bm{H}x)(a)
    \end{pmatrix}, \nonumber 
\end{align}
where the function $x:[0,\infty)\times[a,b]\rightarrow\mathbb{R}^n$ is the state, $x_0\in L^2([a,b],\mathbb{R}^n)$ the initial condition, $\bm{H}:[a,b]\rightarrow\mathbb{R}^{n\times n}$ the Hamiltonian density and a matrix $\tilde{\bm{W}}_B\in\mathbb{R}^{n\times 2n}$ for the prescription of boundary conditions. {For the precise assumptions on the functions and matrices in \eqref{pH} we refer to Assumption~\ref{assumption}.}


The study of stability properties of distributed parameter pH systems is still a recent topic \cite{zwart2010well, le2005dirac,  gernandt2024stability,mora2023exponential,Vill07} {and the results} are often used in the context of stabilization of distributed pH systems  \cite{ramirez2017stabilization, ramirez2014exponential, augner2016stabilisation,GERNANDT2025106030}.

{We focus on the exponential stability of the pH system~\eqref{pH}. In contrast to  classical sufficient conditions via Lyapunov functions~\cite{jacob2012linear} and via the  Gearhart-Prüss theorem
~\cite{engel2000one}, which are often hard to verify, it is known,  see e.g.\ \cite{jacob2012linear}, that for pH systems~\eqref{pH} there are simple sufficient conditions that require a uniform boundary damping on one of the end-points of the interval $[a,b]$ in the boundary conditions. This  sufficient condition is stated in Theorem~\ref{thm: birgit thm} below.  For a vibrating string, this condition is fulfilled if we fix the left end-point at $a$ and connect a damper to the right end-point at $b$. However, when this vibrating string is viewed as being composed of two smaller segments with natural coupling conditions for the forces and velocities at the interconnection point then surprisingly the sufficient condition from \cite{jacob2012linear} on the boundary damping fails to hold. } Therefore, the condition from \cite{jacob2012linear} is not necessary. Only recently a~characterization of exponential stability for pH systems of the form~\eqref{pH} is obtained~\cite{trostorff2022characterisation}. Based on this, the asymptotic stability of such pH systems was characterized in~\cite{waurick2022asymptotic}. However, the conditions given in these works are difficult to apply when the Hamiltonian $\bm{H}$ is not constant.

In the present work, we consider an integrated version of the boundary damping condition from~\cite{jacob2012linear}, which leads to a new characterization of exponential stability. 
To prove this characterization, we revisit 
energy methods that are used in the context of PDEs and in particular wave equations, see e.g.\ 
\cite[Chapter 2]{evans2022partial}, and 
develop and refine local energy methods which 
were, to some extent, already used in the literature on pH systems in 
\cite{Vill07,jacob2012linear,cox1994rate} to derive the sufficient condition in terms of the boundary damping. 
{By analyzing the local energy content in regions of space-time we obtain}
\begin{itemize}
    \item [\rm (i)] an equivalent condition to exponential stability of~\eqref{pH} in terms of the energy content at the boundary;
    \item[\rm (ii)] a proof that there is no short-time decay for the class \eqref{pH} {when $\bm{G}_0$ is skew-symmetric}; 
    \item [\rm (iii)] {present a~network of four strings, where the condition from \cite{jacob2012linear} cannot be applied, but ours can be.}
\end{itemize}

The paper is structured as follows. In Section~\ref{sec:prelim}, we recall preliminary results on the solution of~\eqref{pH} via the semigroup approach. In Section~\ref{sec:energy}, we introduce the two local energy methods and analyze the properties of the functions~$\bm{F}$ and~$\bm{G}$. These properties are used in Section~\ref{sec:long_and_short} to prove a new characterization of exponential stability and to give a negative result in the context of short-time behavior of distributed pH systems~\eqref{pH}. In Section~\ref{sec:vibrating_strings} the characterization of exponential stability is illustrated for {four} interconnected vibrating strings. We conclude the paper in Section~\ref{sec:conclusion}.

\section{Notations and preliminaries}
\label{sec:prelim}
We denote by $\textbf{0}_n$ the zero matrix and by $\textbf{1}_n$ the identity matrix of size $n\times n$. Given $\mathbb{K} = \mathbb{R}$ or $\mathbb{K} = \mathbb{C}$ we consider $\mathbb{K}^n$ equipped with the Euclidean norm which is denoted by $\|\cdot\|$. For a vector $x\in\mathbb{C}^n$ we denote by $x^*$ the  complex-conjugate transposed vector. Further,  $L^2([a,b],\mathbb{K}^n)$ denotes the Lebesgue space of measurable and square integrable functions $f: [a, b] \longrightarrow \mathbb{K}^n$ equipped with the norm $\|f\|_{L^2} = \big(\int\limits_a^b \|f(\zeta) \|^2 \mathrm{d}\zeta\big)^{\frac{1}{2}}$ and $H^1([a,b],\mathbb{K}^n)$ the corresponding Sobolev space of weakly differentiable functions whose  derivative is in $L^2([a,b],\mathbb{K}^n)$. {Given two Hermitian matrices $M_1, M_2 \in \mathbb{C}^{n \times n}$, we write $M_1 \le M_2$ if $\langle x, M_1 x\rangle \le \langle x, M_2x\rangle$ for all $x\in \mathbb{C}^n$.} 

To the right-hand side of the pH system~\eqref{pH}, we associate an operator $\bm{A} : D(\bm{A}) \longrightarrow L^2([a, b], \mathbb{K}^n)$
\begin{align}
\label{eq:def_A}
\begin{split}
    \bm{A} &= \left(\bm{P}_1 \frac{\partial}{\partial \zeta} + \bm{G}_0\right) \bm{H},\\ 
\!D(\bm{A}) &= \Biggl\lbrace  x \in L^2([a,b],\mathbb{K}^{n}) \ \!\big| \ \!\bm{H}x\!\in\!H^1([a,b],\mathbb{K}^{n}),\\  & \qquad\quad  \tilde{\bm{W}}_B\begin{pmatrix}
    (\bm{H}x)(b) \\ (\bm{H}x)(a)
\end{pmatrix}=0 \biggr\rbrace \subseteq L^2([a, b], \mathbb{K}^n)
\end{split}
\end{align}
for some matrix $\tilde{\bm{W}}_B\in\mathbb{R}^{n\times 2n}$.

We impose common assumptions on the system coefficients, see e.g.\ \cite[Assumption 9.1.1]{jacob2012linear}.
\begin{assumption}\label{assumption}
{The coefficients of the pH system~\eqref{pH} have the following properties.}
    \begin{itemize}
        \item[\rm (a)] {The matrix $\tilde{\bm{W}}_B \in \mathbb{R}^{n \times 2n}$ has rank $n$ and fulfills
$$        \tilde{\bm{W}}_B \begin{pmatrix}
            \bm{P}_1^{-1} & \bm{0}_n \\
            \bm{0}_n & -\bm{P}_1^{-1}
        \end{pmatrix}\tilde{\bm{W}}_B^* \ge 0$$.}
        \item[\rm (b)] $\bm{H}\in C^1([a,b],\mathbb{R}^{n\times n})$, $\bm{H}(\zeta) \in \mathbb{R}^{n\times n}$ is symmetric with $m\bm{1}_n\leq \bm{H}(\zeta)\leq M\bm{1}_n$ for all $\zeta \in [a, b]$ and some $M,m > 0$; 
        \item[\rm (c)] $\bm{P}_1\in\mathbb{R}^{n\times n}$ is symmetric and invertible and $\bm{G}_0\in\mathbb{R}^{n\times n}$ is dissipative meaning that $x^*(\bm{G}_0+\bm{G}_0^\top) x \le 0$ for any $x\in \mathbb{K}^n$.
    \end{itemize}
\end{assumption}
It was shown in Theorem~{7.2.4} of \cite{jacob2012linear} that Assumption~\ref{assumption} guarantees that the operator $\bm{A}$ given by \eqref{eq:def_A} generates a contraction semigroup $(\bm{T}(t))_{t\geq 0}$ on the Hilbert space 
\begin{align*}
    \mathcal{X}=L^2([a,b],\mathbb{K}^n),\quad \langle f,g\rangle_{\mathcal{X}}=\frac{1}{2}\int_a^bg(\zeta)^*\bm{H}(\zeta)f(\zeta)\mathrm{d}\zeta. 
\end{align*}

In particular, the function $x(t)=\bm{T}(t)x_0$ is a~\emph{classical solution} of \eqref{pH} whenever $x_0=x(0)\in D(\bm{A})$ meaning that $x\in C^1([0,\infty),\mathcal{X})$, $x(t)\in D(\bm{A})$ and $x$ satisfies \eqref{pH} pointwise for all $t\geq 0$.
Moreover {Assumption~\ref{assumption}} guarantees that for a given initial condition $x_0\in D(\bm{A})$, the classical  solution is unique
and is non-increasing in the weighted norm
\begin{align}
\label{eq:decreasing_solutions}
    \| x(T) \|_{\mathcal{X}}^2\leq     \| x(S) \|_{\mathcal{X}}^2, \quad \text{$0\leq S\leq T<\infty$.}
\end{align}

A semigroup $(\bm{T}(t))_{t \ge 0}$ is called \emph{exponentially stable} if there are constants $C, \omega > 0$ such that 
$\|\bm{T}(t)\| \le C e^{- \omega t}$ for all  $t \ge 0$. We recall a sufficient condition for exponential stability from 
\cite{jacob2012linear, ramirez2014exponential}.
\begin{theorem}\label{thm: birgit thm}
    Consider a pH system~\eqref{pH} that satisfies Assumption~\ref{assumption} with $(\bm{T}(t))_{t \ge 0}$ being the semigroup generated by $\bm{A}$. {If there exists $k > 0$ and $\delta\in\{a, b\}$ such that
    \begin{equation}\label{eqn: birgit cond}
        \langle x, \bm{A} x \rangle_{\mathcal{X}} + \langle  \bm{A}x, x\rangle_{\mathcal{X}} \le -k \|\bm{H}(\delta) x(\delta)\|^2,
    \end{equation}
    holds for all $ x\in D(\bm{A})$
   then $(\bm{T}(t))_{t\ge 0}$ is exponentially stable. }
\end{theorem}
{Theorem~\ref{thm: birgit thm} provides a powerful tool to the stability analysis of pH systems, since it only depends on the boundary matrix $\tilde{W}_B$, but not on the matrices $\bm{G}_0, \bm{P}_1$ and the Hamiltonian. However it is only a sufficient but not a necessaey condition.}

\section{Energy methods revisited}
\label{sec:energy}

\begin{figure}[t]
\centering
\resizebox{0.75\linewidth}{!}{%
\begin{tikzpicture}[
    x=0.8cm,
    y=1cm,
    line cap=round,
    line join=round,
    thick
]

  \def\a{0}
  \def\b{8}
  \def\ymin{0}
  \def\ymax{6}

  \pgfmathsetmacro{\sig}{2.2}
  \pgfmathsetmacro{\ta}{4.0}
  \pgfmathsetmacro{\gam}{0.8}

  \pgfmathsetmacro{\Sval}{\sig-\gam}
  \pgfmathsetmacro{\Tval}{\ta+\gam}
  \pgfmathsetmacro{\up}{\ta+0.2*\gam}
  \pgfmathsetmacro{\do}{\sig-0.2*\gam}

  \pgfmathsetmacro{\ze}{0.2*(\a+\b)}
  \pgfmathsetmacro{\midL}{0.5*(\sig+\ta)}
  \pgfmathsetmacro{\midR}{0.5*(\Sval+\Tval)}

  \def\tick{0.25}
  \def\leftpad{14pt}

  \draw (\a,\ymin) -- (\b,\ymin);

  \node[below] at (\ze,\ymin) {$\zeta$};
  \node[below] at (\a,\ymin) {$a$};
  \node[below] at (\b,\ymin) {$b$};

  \draw (\a,\ymin) -- (\a,\ymax);
  \draw (\b,\ymin) -- (\b,\ymax);

  \draw[line width=2pt]
    (\ze,\do) -- (\ze,\up)
    node[midway,right=6pt] {
      $\bm{F}(\zeta)=
      \!\!\!\!\!\!
      \int\limits_{\sigma-\gamma(\zeta-a)}
                  ^{\tau+\gamma(\zeta-a)}
      (x^*\bm{H}x)(t,\zeta)\,\mathrm{d}t$
    };

  \draw[line width=2pt]
    (\a,\sig) -- (\a,\ta);

  \draw[line width=2pt]
    (\b,\Sval) -- (\b,\Tval);

  \draw (\a-\tick,\sig) -- (\a+\tick,\sig)
    node[left=\leftpad] {$\sigma$};

  \draw (\a-\tick,\ta) -- (\a+\tick,\ta)
    node[left=\leftpad] {$\tau$};

  \node[right=4pt] at (\a,\midL)
    {$\bm{F}(a)$};

  \node[left=4pt] at (\a,\midL)
    {$t$};

  \draw (\b-\tick,\Sval) -- (\b+\tick,\Sval)
    node[right=2pt] {
      \!\!\!$\begin{matrix}
        S=~~~~~~~~~~~\\
        \sigma-\gamma(b-a)~~~
      \end{matrix}$
    };

  \draw (\b-\tick,\Tval) -- (\b+\tick,\Tval)
    node[right=0pt] {
      \!\!\!$\begin{matrix}
        T=~~~~~~~~~~~\\
        \tau+\gamma(b-a)~~~
      \end{matrix}$
    };

  \node[right=8pt] at (\b,\midR)
    {$\bm{F}(b)$};

  \draw[dotted]
    (\a,\Tval) -- (\b,\Tval)
    node[midway,above] {
      $\|x(T)\|^2_{\mathcal{X}}
      =\tfrac{1}{2}
      \int_a^b
      (x^*\bm{H}x)(T,\zeta)
      \,\mathrm{d}\zeta$
    };

  \draw[dotted]
    (\a,\Sval) -- (\b,\Sval)
    node[midway,below] {
      $\|x(S)\|^2_{\mathcal{X}}
      =\tfrac{1}{2}
      \int_a^b
      (x^*\bm{H}x)(S,\zeta)
      \,\mathrm{d}\zeta$
    };

  \draw[dashed]
    (\a,\sig) -- (\b,\Sval);

  \draw[dashed]
    (\a,\ta) -- (\b,\Tval);

\end{tikzpicture}%
}

\caption{
Visual sketch of the function $\bm{F}$.
The function $\tilde{\bm{F}}$ can formally
be obtained by flipping the sketch horizontally.
}
\label{fig:proof}

\end{figure}

\begin{figure}[t]
\centering
\resizebox{0.75\linewidth}{!}{%
\begin{tikzpicture}[
    x=1cm,
    y=1cm,
    line cap=round,
    line join=round,
    thick
]

  \def\a{0}
  \def\b{8}
  \def\ymin{0}
  \def\ymax{6}

  \pgfmathsetmacro{\eps}{2}
  \pgfmathsetmacro{\tatwo}{5}
  \pgfmathsetmacro{\taone}{2.5}
  \pgfmathsetmacro{\alp}{3}
  \pgfmathsetmacro{\bet}{5}
  \pgfmathsetmacro{\zet}{4}

  \def\tick{0.25}
  \def\leftpad{14pt}

  \draw (\a,\ymin) -- (\b,\ymin)
    node[midway,below=6pt] {$\zeta$};

  \node[below] at (\a,\ymin) {$a$};
  \node[below] at (\b,\ymin) {$b$};

  \draw (\a,\ymin) -- (\a,\ymax);
  \draw (\b,\ymin) -- (\b,\ymax);

  \draw[dashed]
    (\a+\eps,0) -- (\a,\tatwo+1);

  \draw[dashed]
    (\b-\eps,0) -- (\b,\tatwo+1);

  \node[left=1pt] at (\a,\tatwo+1)
    {$\frac{\varepsilon}{\gamma}$};

  \node[left=1pt] at (\a,\taone+1)
    {$\frac{\beta-\alpha}{2\,\gamma}$};

  \node[left=1pt] at (\a,\tatwo/4)
    {$t$};

  \node[below] at (\a+\eps,0)
    {$a+\varepsilon$};

  \node[below] at (\b-\eps,0)
    {$b-\varepsilon$};

  \draw[dashed]
    (\alp,0) -- (\zet,\taone+1);

  \draw[dashed]
    (\bet,0) -- (\zet,\taone+1);

  \draw[loosely dotted]
    (\a,\tatwo+1) -- (\b,\tatwo+1);

  \draw[loosely dotted]
    (\a,\taone+1) -- (\zet,\taone+1);

  \draw[line width=2pt]
    (\a,0) -- (\a+\eps,0)
    node[midway,above] {$\tilde{\bm{G}}_1(0)$};

  \draw[line width=2pt]
    (\b-\eps,0) -- (\b,0)
    node[midway,above] {$\tilde{\bm{G}}_2(0)$};

  \draw[line width=2pt]
    (\alp,0) -- (\bet,0)
    node[midway,above] {$\bm{G}(0)$};

  \draw[line width=2pt]
    (\b,\tatwo/4) --
    (\b-\eps*3.9/5,\tatwo/4)
    node[midway,above] {$\tilde{\bm{G}}_2(t)$};

  \draw[line width=2pt]
    (\a,\tatwo/4) --
    (\a+\eps*3.9/5,\tatwo/4)
    node[midway,above] {$\tilde{\bm{G}}_1(t)$};

  \draw[line width=2pt]
    (0.62*\alp + 0.38*\zet,0.5*\taone) --
    (0.62*\bet + 0.38*\zet,0.5*\taone)
    node[midway,above=-2pt] {$\bm{G}(t)$};

  \node[below] at (\alp,0) {$\alpha$};
  \node[below] at (\bet,0) {$\beta$};

\end{tikzpicture}%
}

\caption{
Visual sketch of the functions $\bm{G}$ and
$\tilde{\bm{G}}=\tilde{\bm{G}}_1+\tilde{\bm{G}}_2$.
}
\label{fig:proof2}

\end{figure}

In the following, we obtain two energy methods and derive consequences of these.
We start by revisiting an energy estimate for \eqref{pH} that was used to prove Theorem~\ref{thm: birgit thm} in \cite{VilZLGM09}, see also \cite{cox1994rate,jacob2012linear}. To this end, we use the shorthand
\[
(x^* \bm{H} x)(t, \zeta):=x(t,\zeta)^* \bm{H}(\zeta) x(t, \zeta).
\]
Consider some constants $\sigma, \tau, \gamma >0 $ with $\sigma<\tau$ and a solution 
{
$x: [0, \infty) \rightarrow \mathcal{X}$ with $x(t) \in H^1([a, b]; \mathbb{K}^n)$ for any $t \ge 0$
} of
\begin{equation}\label{eqn: no boundary}
    \frac{\mathrm{d}}{\mathrm{d}t} x(t, \zeta) = \left(\bm{P}_1 \frac{\partial}{\partial \zeta} + \bm{G}_0\right)(\bm{H}(\zeta) x(t, \zeta)).
\end{equation} Equation \eqref{eqn: no boundary} is understood to be fulfilled pointwise for almost every $\zeta \in [a, b]$.
We define
\begin{equation}\label{technical time}
    \begin{split}
        &\bm{F}
        (\zeta) = \int\limits_{\sigma - \gamma (\zeta - a)}^{\tau + \gamma (\zeta - a)} (x^* \bm{H} x)(t, \zeta) \ \mathrm{d}t,\quad \\
        &\Tilde{\bm{F}}
        (\zeta) 
        = \int\limits_{\sigma + \gamma(\zeta-b) }^{\tau - \gamma (\zeta-b)} (x^* \bm{H} x)(t, b +a - \zeta) \ \mathrm{d}t
    \end{split}
\end{equation} 
for $\zeta\in[a,b]$.
In order to have positive integration bounds in \eqref{technical time}, we assume that  $\sigma\geq \gamma(b-a)$ holds. 
The function $\bm{F}$ is depicted in Figure~\ref{fig:proof}. The function $\tilde{\bm{F}}$ can be obtained by the change of variables $\tilde{\bm{F}}(\zeta)=\bm{F}(a+b-\zeta)$.

The reason why we do not assume $x(t) = \bm{T}(t) x_0$ for $x_0 \in D(\bm{A})$ in \eqref{technical time} is that we will apply the following lemma in subsequent arguments without the homogeneous boundary conditions assumed in~\eqref{pH}.
Note that $\bm{F}$ and $\Tilde{\bm{F}}$ were already considered in~\cite{jacob2012linear} for $\sigma - \gamma{ (b - a) }= 0$. In that work, $\tau > \sigma$ was fixed, see the proof of Lemma~9.1.2. By following the lines of that proof we obtain the following more general result.
\begin{proposition}
\label{lem: side ways}
Consider \eqref{eqn: no boundary} with $\bm{H}, {\bm{P}_1, \bm{G}_0}$ satisfying Assumption~\ref{assumption} (b) and (c) and assume that $\gamma > 0$ is large enough such  that
$\gamma \bm{H}(\zeta) \ge \pm \bm{P}_1^{-1}$, for all $\zeta \in [a, b].$ 
Then there {is a constant} $\kappa > 0$ such that the following properties hold. 
Given $\sigma < \tau$ satisfying $\sigma\geq \gamma(b-a)$ and $x \in C^1([0, \infty) , H^1([a, b], \mathbb{K}^{n}))$ being a solution of \eqref{eqn: no boundary}, the functions $\bm{F}$ and $\Tilde{\bm{F}}$ given by \eqref{technical time}  have the following properties:
\begin{itemize}
    \item[\rm (a)] $\bm{F}(\zeta)\geq0$ and $\Tilde{\bm{F}}(\zeta)\geq 0$ for all $\zeta\in[a,b]$;
    \item[\rm (b)] $\bm{F}$ and $ \Tilde{\bm{F}}$ satisfy for all  $\zeta\in[a,b]$
    \begin{equation}\label{energy time}
            \!\!\! \!\!\! \!\!\!
                \frac{\rm d}{{\rm d}\zeta} \bm{F}(\zeta) \ge - \kappa \bm{F}(\zeta),\quad 
                \frac{\rm d}{{\rm d}\zeta}\Tilde{\bm{F}}(\zeta) \le \kappa \Tilde{\bm{F}}(\zeta);
        \end{equation} 
    \item[\rm (c)] The function $\zeta \mapsto e^{\kappa \zeta} \bm{F}(\zeta)$ is monotonically increasing on $[a,b]$ whereas $\zeta \mapsto e^{-\kappa \zeta}\Tilde{\bm{F}}(\zeta)$ is monotonically decreasing on $[a, b]$.
\end{itemize}
{
\begin{remark}
    Item \rm (c) of Proposition~\ref{lem: side ways} implies 
    \begin{align}
        \int\limits_{\sigma}^\tau (x^*\bm{H}x)(t, a) \, \mathrm{d}t &\le \mathrm{e}^{\kappa (b-a)}\int\limits_{S}^T (x^*\bm{H}x)(t, b) \, \mathrm{d}t \label{eqn: side left right} \\
        \int\limits_{\sigma}^\tau (x^*\bm{H}x)(t, b) \, \mathrm{d}t &\le \mathrm{e}^{\kappa (b-a)}\int\limits_{S}^T (x^*\bm{H}x)(t, a) \, \mathrm{d}t. \label{eqn: side right left}
    \end{align}
\end{remark}
}
\end{proposition}
Since Proposition~\ref{lem: side ways} is only a technical modification of the first step in the proof of Lemma~9.1.2 in \cite{jacob2012linear}, {we omit the proof here.}

In the following proposition, we generalize Lemma~9.1.2 from \cite{jacob2012linear}, but we use the squared norm $v^* \bm{H}(\zeta)v$ instead of $\|\bm{H}(\zeta) v\|^2$, $v \in \mathbb{K}^n$, $\zeta \in [a, b]$ because this one is easier to use in the applications considered later. That is not a significant modification, since all norms on $\mathbb{K}^n$ are equivalent.

\begin{proposition}\label{energy time theorem}
    Consider a pH system~\eqref{pH} that satisfies Assumption~\ref{assumption} and let $(\bm{T}(t))_{t \ge 0}$ be the contraction semigroup generated by $\bm{A}$. Let $\sigma,\tau,\gamma>0$ such that $\gamma (b-a) \le \sigma < \tau $ holds with 
    $\gamma \bm{H}(\zeta) \ge \pm \bm{P}_1^{-1}$, for all $\zeta \in [a, b]$, and set \[0\leq S := \sigma - \gamma (b-a)\leq  \tau + \gamma(b-a)=: T.\] Then there are constants $c, d > 0$ such that the following conditions hold.
    \begin{enumerate} 
        \item[\rm (a)]  For all $x_0\in D(\bm{A})$ and $\delta\in\{a,b\}$ the solution $x(t) = \bm{T}(t) x_0$ of \eqref{pH} fulfills
\begin{equation}\label{energy time simple conclusion}
             \begin{split}
                 &\| x(T) \|_{\mathcal{X}}^2 \le c \int\limits_S^T  (x^*\bm{H} x)(t, \delta) \ \mathrm{d}t.
             \end{split}
        \end{equation}
        \item[\rm (b)] For all $x_0\in D(\bm{A})$ and $\delta\in\{a,b\}$ the solution $x(t) = \bm{T}(t) x_0$ of \eqref{pH} fulfills
        \begin{equation}
        \label{simple conclusion +1}
            \begin{split}
    &\int\limits_\sigma^\tau  (x^*\bm{H} x)(t, \delta) \mathrm{d}t\le d \ \|x(S)\|_{\mathcal{X}}^2.
            \end{split}
        \end{equation}
    \end{enumerate}
\end{proposition}
\begin{proof}
The assertion {\rm (a)} was already obtained in the proof of Lemma $9.1.2$ of \cite{jacob2012linear}.  

    To prove (b), {fix a classical solution $x(t) = \bm{T}(t)x_0$ with $x_0 \in D(\bm{A})$ so that Proposition~\ref{lem: side ways} can be applied.} Estimate \begin{align*} 
    \!\!\!\!
        2\|x(S)\|_{\mathcal{X}}^2&=\int\limits_a^b(x^*\bm{H}x)(S, \zeta) \, \mathrm{d}\zeta \\&\stackrel{\eqref{eq:decreasing_solutions}}{\ge} \tfrac{1}{T - S}\int\limits_{S}^T \int\limits_a^b (x^*\bm{H}x)(t,\zeta)  \ \mathrm{d}\zeta\ \mathrm{d}t \\
        &\ge \tfrac{1}{T-S} \int\limits_a^b\underbrace{\int\limits_{\sigma - \gamma (\zeta-a)}^{\tau + \gamma (\zeta-a)} (x^*\bm{H}x)(t,\zeta)\,\mathrm{d}t}_{=\bm{F}(\zeta)} \, \mathrm{d}\zeta\\
       &\stackrel{\text{Prop.~\ref{lem: side ways}}}{\ge} \tfrac{1}{T- S}\int\limits_a^b e^{\kappa(a-\zeta)}\bm{F}(a) \, \mathrm{d} \zeta \\
        &= \tfrac{1}{T-S}\int\limits_a^b e^{\kappa(a-\zeta)}\, \mathrm{d} \zeta \int\limits_\sigma^\tau (x^*\bm{H}x)(t, a)\, \mathrm{d} t.
    \end{align*}
    Rearranging terms leads to \eqref{simple conclusion +1}. 
The other bound for $\delta = b$ can be proven in a similar way using the properties of $\Tilde{\bm{F}}$.
\end{proof}

{
\begin{remark}
    Note that \eqref{energy time simple conclusion} implies that the system 
    \begin{equation}\label{eqn: adm/obs}
        \begin{split}
            \frac{\mathrm{d}}{\mathrm{d}t} x(t) = \bm{A}x(t) \\
        y(t) = \bigl(\bm{C}_\delta x\bigr)(t) = x(t, \delta)
        \end{split}
    \end{equation}
    for $\delta \in \{a, b\}$
    is final-state observable. On the other hand \eqref{simple conclusion +1} is related to the admissibility of \eqref{eqn: adm/obs}.
\end{remark}
}

We will now study another local energy method for equation \eqref{pH} which is based on considering the energy for fixed times and integrating over a certain spatial domain. 

Suppose that $x_0 \in D(\bm{A})$ and $x(t) = \bm{T}(t) x_0$ and $a < \alpha < \beta < b$, $\varepsilon > 0$,
define \begin{align}\label{eqn: G middle}
    \bm{G}(t) &=\!\!\!\!\!\int\limits_{\alpha + \gamma t}^{\beta - \gamma t} \!\!(x^* \bm{H} x)(t, \zeta) \, \mathrm{d} \zeta, \,\,\, 0 \le t \le \frac{\beta - \alpha}{2 \gamma},\\
\label{eqn: G bound}
        \Tilde{\bm{G}}(t) &= \!\!\!\!\!\!\int\limits_{a}^{a + \varepsilon - \gamma t} \!\!\!\!(x^* \bm{H} x)(t, \zeta) \, \mathrm{d} \zeta + \!\!\!\!\int\limits_{b - \varepsilon + \gamma t}^b \!\!(x^*\bm{H}x)(t, \zeta) \, \mathrm{d} \zeta
\end{align}
for $0\le t \le \frac{\varepsilon}{\gamma}$. 
The functions $\bm{G}$ and $\tilde{\bm{G}}$ are depicted in Figure~\ref{fig: proof 2}.

\begin{proposition}
\label{energy spatial theorem}
Consider a pH system~\eqref{pH} that satisfies Assumption~\ref{assumption} with $(\bm{T}(t))_{t \ge 0}$ being the semigroup generated by $\bm{A}$.
Then there is $\gamma > 0$ such that given any $x_0 \in D(\bm{A})$, $\bm{G}$ and $\Tilde{\bm{G}}$ defined by \eqref{eqn: G middle} and \eqref{eqn: G bound}  have the following properties.
\begin{itemize}
    \item[\rm (a)] $\bm{G}(t) \ge 0, \ 0 \leq t \le \frac{\beta - \alpha}{2 \gamma}$ and  $\Tilde{\bm{G}}(t) \ge 0, \ 0\leq t\le\frac{\varepsilon}{\gamma}$;
    \item[\rm (b)]
        $\tfrac{\mathrm{d}}{\mathrm{d}t} \bm{G}(t) \le 0,\ 0\leq t\le \frac{\beta - \alpha}{2 \gamma}$ and 
        $\tfrac{\mathrm{d}}{\mathrm{d}t} \Tilde{\bm{G}}(t) \le 0, \ 0\leq t\le \frac{\varepsilon}{\gamma}$;
    \item[\rm (c)]
        $\bm{G}(s) \geq \bm{G}(t), \ 0\leq s\leq t\leq \frac{\beta - \alpha}{2 \gamma}$ and \\
        $\Tilde{\bm{G}}(s) \geq \Tilde{\bm{G}}(t), \ 0\leq s\leq t\leq \frac{\varepsilon}{\gamma}.$
        \end{itemize}
\end{proposition}
    \begin{proof}
    Fix $\gamma > 0$ that will be chosen later.
    The assertion \rm (a) follows from the nonnegativity of the integrands. 
       To prove the first estimate in (b), we calculate \begin{align}\nonumber
                    \frac{\mathrm{d}}{\mathrm{d}t} \bm{G}(t) &= -\gamma (x^*\bm{H}x)(t, \beta - \gamma t) - \gamma (x^*\bm{H}x)(t, \alpha + \gamma t)\\ &\quad + \int\limits_{\alpha + \gamma t}^{\beta - \gamma t} \frac{\mathrm{d}}{\mathrm{d}t} (x^*\bm{H}x)(t, \zeta) \, \mathrm{d} \zeta. \label{calc}
            \end{align}
            Evaluation of the integral term leads to
            \begin{align*}
            \int\limits_{\alpha + \gamma t}^{\beta - \gamma t} \frac{\mathrm{d}}{\mathrm{d}t} (x^*\bm{H}x)(t, \zeta) \, \mathrm{d} \zeta 
            \quad & {\le} \int\limits_{\alpha + \gamma t}^{\beta - \gamma t} \frac{\partial}{\partial \zeta} (x^*\bm{H}\bm{P}_1 \bm{H}x)(t, \zeta) \, \mathrm{d} \zeta \\ 
            &= \left[(x^*\bm{H}\bm{P}_1\bm{H}x)(t,\zeta)\right]_{\zeta=\alpha + \gamma t}^{\zeta=\beta - \gamma t}.
            \end{align*}
            {where we used that $\bm{G}_0$ is dissipative.}
            Invoking \eqref{calc} yields \begin{align*}
                \frac{\mathrm{d}}{\mathrm{d}t} \bm{G}(t) &= (x^*(\bm{H}\bm{P}_1\bm{H} - \gamma \bm{H})x)(t, \beta - \gamma t) \\ &\quad + (x^*(-\bm{H}\bm{P}_1\bm{H} - \gamma \bm{H})x)(t, \alpha + \gamma t)\\
                            &\leq 0,
            \end{align*}
            when choosing $\gamma$ large enough so that 
            \begin{align}
\label{choice of gamma for G}                
            \pm \bm{H}(\zeta)\bm{P}_1\bm{H}(\zeta) - \gamma \bm{H}(\zeta) \le 0, \, \text{ $\zeta \in [a, b]$.}
            \end{align}
            
            To prove the second estimate, we calculate with $\gamma>0$ chosen such that \eqref{choice of gamma for G} holds,  
            \begin{align*}
                \frac{\mathrm{d}}{\mathrm{d}t} \Tilde{\bm{G}}(t)&{\le} 
                     -\gamma (x^*\bm{H}x)(t,a+\varepsilon - \gamma t) \\ &\quad - \gamma (x^*\bm{H}x)(t,b-\varepsilon + \gamma t)\\ &\quad + \int\limits_a^{a +\varepsilon-\gamma t}\frac{\partial}{\partial \zeta} (x^*\bm{H}\bm{P}_1\bm{H}x)(t, \zeta) \, \mathrm{d} \zeta \\
                    &\quad  + \int\limits_{b -\varepsilon + \gamma t}^{b}\frac{\partial}{\partial \zeta} (x^*\bm{H}\bm{P}_1\bm{H}x)(t, \zeta) \, \mathrm{d} \zeta \\
                    = &(x^*(\bm{H}\bm{P}_1\bm{H} - \gamma \bm{H})x)(t,a + \varepsilon - \gamma t) \\& + (x^*(-\bm{H}\bm{P}_1\bm{H} - \gamma \bm{H})x)(t,b - \varepsilon + \gamma t) \\
                    &+ \left[(x^*\bm{H}\bm{P}_1\bm{H}x)(t, \zeta)\right]_{\zeta=a}^{\zeta=b}\\ \leq &0,
            \end{align*}
            where we used in the last step \eqref{choice of gamma for G}, and 
 $\left[(x^*\bm{H}\bm{P}_1\bm{H}x)(t, \zeta)\right]_{\zeta=a}^{\zeta=b} \le 0$.  This proves the assertion (b). The property (c) immediately follows from (b).
    \end{proof}

\section{Applications to long- and short-time behavior}
\label{sec:long_and_short}

As a first application of the local energy method {in Proposition~\ref{energy time theorem}}, we obtain a characterization for exponential stability of the pH system~\eqref{pH} which is a generalization of Theorem~\ref{thm: birgit thm} from \cite{jacob2012linear,Vill07}. 
\begin{theorem}
\label{stability basic}
Consider a pH system~\eqref{pH} which satisfies Assumption~\ref{assumption} and let $(\bm{T}(t))_{t \ge 0}$ be  the contraction semigroup generated by $\bm{A}$. Furthermore assume that $\gamma > 0$ is large enough that
    $\gamma \bm{H}(\zeta) \ge \pm \bm{P}_1^{-1}$, for all  $\zeta \in [a, b]$. Then the  following assertions are equivalent: 
    \begin{itemize}   
        \item[\rm (i)] 
        $(\bm{T}(t))_{t \ge 0}$ is exponentially stable.
        \item[\rm (ii)] {
        There exist $0 < \sigma < \tau < T$ satisfying $\tau - \sigma > 2\gamma (b-a)$, $k > 0$ and $\delta \in \{a, b\}$ such that
        \begin{equation}\label{eqn: better condition}
        \|x(T)\|_{\mathcal{X}}^2 - \|x_0\|^2_{\mathcal{X}} \le -k \int\limits_\sigma^\tau(x^* \bm{H}x)(t, \delta) \, \mathrm{d} t
        \end{equation}
        for all $x_0 \in D(\bm{A})$ and $x(t) = \bm{T}(t) x_0$.
        }
    \end{itemize}
\end{theorem}
    \begin{proof}
{We will first prove that \rm (ii) implies \rm (i).} To estimate the right-hand side in \eqref{eqn: better condition} from above, we derive an upper bound for $\|x(\tau)\|^2_{\mathcal{X}}$ by a suitable application of  Proposition~\ref{energy time theorem}~\rm (a). Note that $S$ and $T$ in that proposition play the role of $\sigma$ and $\tau$ that appear in \eqref{eqn: better condition}. Thus, we use Proposition~\ref{energy time theorem}~\rm(a) with $\sigma'$ and $\tau'$ satisfying $\sigma = \sigma' - \gamma (b-a)$, $\tau = \tau' + \gamma (b-a)$ and
$\tau' - \sigma' > 0$, $\sigma' \ge \gamma (b-a)$
to infer 
$\|x(\tau)\|^2_{\mathcal{X}} \le c \int\limits_\sigma^\tau (x^*\bm{H}x)(t, \delta)\, dt$ for some $c > 0$. But then
            \begin{align*}
                \|x(T)\|_{\mathcal{X}}^2 - \|x_0\|^2_{\mathcal{X}} &\le -k \int\limits_\sigma^\tau (x^* \bm{H} x)(t, \delta) \, \mathrm{d} t \\
                &\le -\frac{k}{c} \|x(\tau)\|_{\mathcal{X}}^2 \stackrel{\eqref{eq:decreasing_solutions}}{\le} -\frac{k}{c} \| x (T) \|^2_{\mathcal{X}}.
            \end{align*}
                It follows that $\|x(T)\|_{\mathcal{X}}^2 \le \frac{c}{c+k}\|x_0\|_{\mathcal{X}}^2$ so that $\|\bm{T}(T)\| < 1$ which implies that the semigroup generated by $\bm{A}$ is exponentially stable, see Proposition~1.2 in Chapter~5 of \cite{engel2000one}.
                
                {
                Conversely suppose that \rm (i) holds. We will show that \rm (ii) holds both for $\delta = a$ and $\delta = b$.
                Using that
                } $\bm{A}$ generates an exponentially stable semigroup, $\|\bm{T}(T)\| < 1$ holds for~$T$ sufficiently large, which implies
                \begin{align}
                    \label{converse_estimate}
            \|x(T)\|_{\mathcal{X}}^2 - \|x_0\|_{\mathcal{X}}^2 \le -\kappa \|x_0\|^2_{\mathcal{X}}
        \end{align}
            for a constant $\kappa > 0$. By applying the contractivity estimate~\eqref{eq:decreasing_solutions}, the inequality \eqref{converse_estimate} remains true for all times larger than $T$. Therefore, we may assume that $T > 4\gamma(b-a)$ holds. 
            Using item (b) of Proposition~\ref{energy time theorem}, we obtain $ \int\limits_\sigma^\tau (x^* \bm{H} x)(t, \delta) \, \mathrm{d} t {\le d \, \|x(\sigma)\|_{\mathcal{X}}^2} \le d \, \|x_0\|_{\mathcal{X}}^2$ with $\sigma = \gamma (b-a)$, $\tau = T - \gamma(b-a)$ so that $\tau - \sigma > 2\gamma(b-a)$. Hence, inserting this into \eqref{converse_estimate} we obtain 
       \eqref{eqn: better condition} {which concludes the proof.}
    \end{proof}

A recent topic in the literature on control theory is the short-time behavior and \emph{hypocoercivity} of dissipative systems, see for example \cite{achleitner2022hypocoercivity, achleitner2025hypocoercivity}, where 
for certain exponentially stable semigroups $(\bm{T}(t))_{t\geq 0}$ it was shown that there are $a \in 2 \mathbb{N}_0 + 1$ and $c_1, c_2, t_0 > 0$ such that for all $t \le t_0$
{
\[
1 - c_1t^a + o(t^{a}) \le \|\bm{T}(t)\|^2 \le 1 - c_2t^a+ o(t^{a}), \ t\rightarrow 0_+.
\]
}

The next proposition shows that pH systems~\eqref{pH} do not obey this type of short-time decay {and we can also estimate the time where the norm of the semigroup remains constantly equal to $1$ in the short-time limit}.
\begin{proposition}
            Consider a pH system~\eqref{pH} {with $\bm{G}_0=-\bm{G}_0^\top$} and which satisfies Assumption~\ref{assumption} and let $(\bm{T}(t))_{t \ge 0}$ be  the contraction semigroup generated by $\bm{A}$. Then there exists $\tau > 0$ such that $\|\bm{T}(t)\| = 1$ for $0\leq t \le \tau$. 
\end{proposition}
\begin{proof}\
Fix $0 < \varepsilon < \frac{1}{4}(b-a)$ and $\gamma > 0$ so that Proposition~\ref{energy spatial theorem} can be applied.
We choose an (arbitrary) initial value $x_0 \in D(\bm{A})$ with $\|x_0\|_{\mathcal{X}} = 1$ that has support in the interval $(a + \frac{1}{4}(b-a), b - \frac{1}{4}(b-a))$. Then the solution of \eqref{pH} is given by  $x(t) = \bm{T}(t) x_0$ and it fulfills
\begin{align*}    
\tilde{\bm{G}}(0)&=\int\limits_a^{a + \varepsilon } (x^*\bm{H}x)(0, \zeta) \, \mathrm{d}\zeta + \int\limits_{b -\varepsilon }^{b} (x^*\bm{H}x)(0, \zeta)\mathrm{d}\zeta\\
&=\int\limits_{a}^{a + \varepsilon} x_0^*\bm{H}(\zeta)x_0 \, \mathrm{d} \zeta + \int\limits_{b -\varepsilon}^b x_0^*\bm{H}(\zeta)x_0\, \mathrm{d} \zeta=0. 
\end{align*}
In combination with  Proposition~\ref{energy spatial theorem}~(a) and (c), we obtain
\[
0\leq \tilde{\bm{G}}(t)\leq \tilde{\bm{G}}(0)=0,\quad 0<t\leq\frac{\varepsilon}{\gamma},
\]
and therefore this leads to \begin{align*}
           \tilde{\bm{G}}(t)&= \!\!\!\!\!\!\int\limits_a^{a + \varepsilon - \gamma t} \!\!\!(x^*\bm{H}x)(t, \zeta) \, \mathrm{d}\zeta + \!\!\!\int\limits_{b -\varepsilon + \gamma t}^{b} \!\!\!(x^*\bm{H}x)(t, \zeta)\mathrm{d}\zeta = 0
        \end{align*}
        for $0 < t < \frac{\varepsilon}{\gamma}$. From the coercivity of $\bm{H}$ it follows that $x(t, a) = x(t, b) = 0$ holds for all $0 < t < \frac{\varepsilon}{\gamma}$ and thus  \begin{align*}
            \|x(t)\|_{\mathcal{X}}^2 - \|x_0\|_{\mathcal{X}}^2 &= \int\limits_0^t 
            \Biggl[ (x^*\bm{H}\bm{P}_1 \bm{H}x)(s, \zeta) \biggr]_{\zeta = a}^{\zeta = b}
            \, \mathrm{d}s= 0
        \end{align*}
        {where we used $\bm{G}_0 = -\bm{G}_0^\top$.
        This implies $\|\bm{T}(t)\| = \sup\limits_{x_0 \in \mathcal
        X\setminus \{0\}} \frac{\|\bm{T}(t)x_0\|}{\|x_0\|} = 1$ if $0 \le t < \frac{\varepsilon}{\gamma}$.}
\end{proof}
\section{Application to a vibrating string network}
\label{sec:vibrating_strings}
Consider four vibrating strings connected as in Figure~\ref{fig:Ex2}. The first string I is fixed at one end and connected to the second string II via a massless bar. The other end of the second string II is connected to another massless bar. The third and fourth strings are also connected to this massless bar, and both of them are connected to a damper to ground on the other end.

\begin{figure}[!htbp]
  \centering
  \includegraphics[width=1\linewidth]{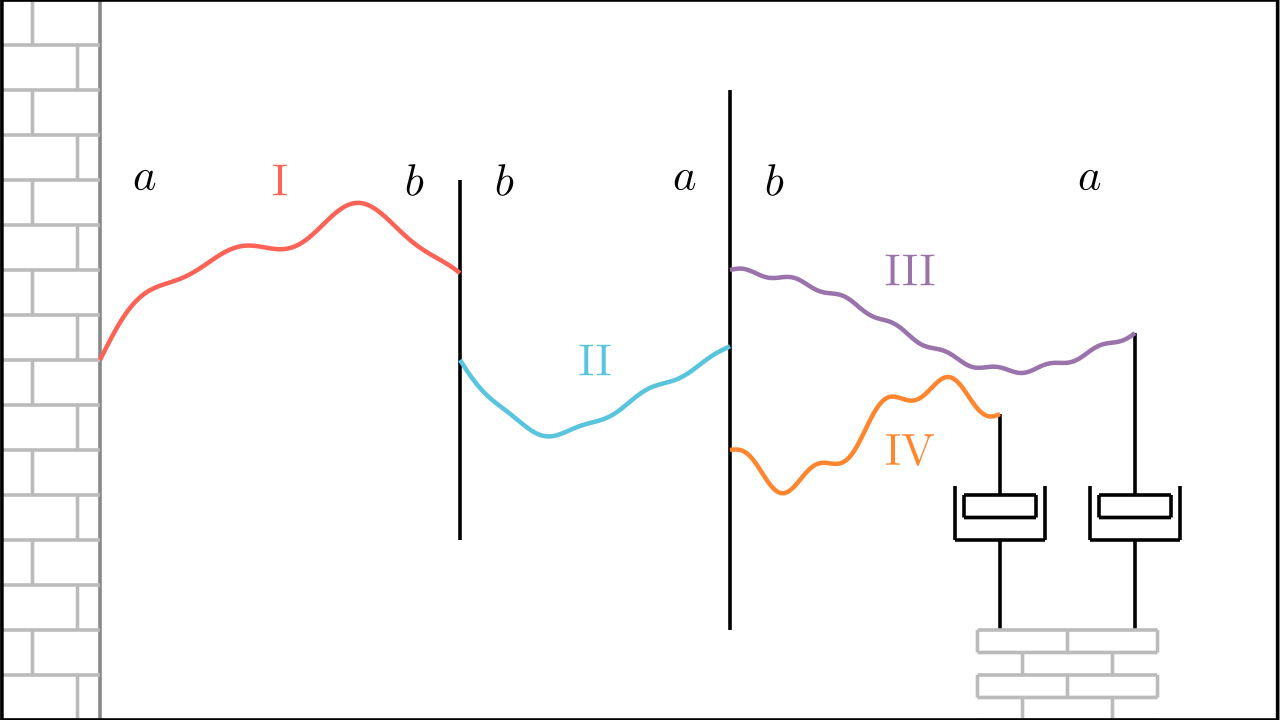}
  \caption{{A network of four interconnected vibrating strings.}}
  \label{fig:Ex2}
\end{figure}

    Mathematically, this means that the pH system~\eqref{pH} is given by  $\bm{P}_1 =\diag(\bm{P}_{1,i})_{i=\RNum{1}}^{\RNum{4}}$ and $\bm{H}=\diag(\bm{H}_i)_{i=\RNum{1}}^{\RNum{4}}$ with 
    \begin{align*}
       \bm{P}_{1,i}=\begin{pmatrix}
            0& 1\\1&0
        \end{pmatrix},\quad 
        \bm{H}_i(\zeta)=\begin{pmatrix}
            \frac{1}{\rho_{i}(\zeta)} & 0\\  0 & T_{i}(\zeta)
        \end{pmatrix},
    \end{align*}
    where $\rho_{i}$ and $T_{i}$ are positive and continuously differentiable for all $i=\RNum{1},\ldots,\RNum{4},$ and $\bm{G}_0 = \bm{0}_8$.
    {The boundary conditions are given by}
    \begin{equation}\label{eqn: bound four string 1}
        \begin{split}
            v_{\RNum{1}}(a) &= 0,\ v_{\RNum{1}}(b) = v_{\RNum{2}}(b),\\
        v_{\RNum{3}}(b) &= v_{\RNum{2}}(a),\ v_{\RNum{4}}(b) = v_{\RNum{2}}(a),
        \end{split}
    \end{equation}
    and 
    \begin{align}\label{eqn: bound four string 2}
        &F_{\RNum{1}}(b) + F_{\RNum{2}}(b) = 0, \quad 
        F_{\RNum{2}}(a) + F_{\RNum{3}}(b) + F_{\RNum{4}}(b)= 0   \\ \nonumber 
        &F_{\RNum{3}}(a) = - \sigma_{\RNum{3}}\, v_{\RNum{3}}(a),\quad 
        F_{\RNum{4}}(a) = - \sigma_{\RNum{4}}\, v_{\RNum{4}}(a), \nonumber 
    \end{align}
    where $\sigma_{\RNum{3}},\sigma_{\RNum{4}}>0$ are damping constants and 
        $v_{i}(\delta) = {\frac{1}{\rho_{i}(\delta)}}x_{2i-1}(\delta)$
    for $\delta \in \{a, b\}$ and for $i\in\{\RNum{1},\RNum{2},\RNum{3},\RNum{4}\}$
    \begin{align*}
        F_{i}(b) = T_{i}(b)x_{2i}(b),\quad F_{i}(a) = -T_{i}(a)x_{2i}(a).
    \end{align*}
    The domain of $\bm{A}$ is given by \[
    D(\bm{A}) = \{ x \in H^1((a, b); \mathbb{K}^8) \, |\, x \text{ satisfies \eqref{eqn: bound four string 1}, \eqref{eqn: bound four string 2}}\}.
    \]
    {Given $x(t) = \bm{T}(t)x_0$ for $x_0 \in D(\bm{A})$} we calculate 
    {
    \begin{equation}\label{eqn: four string simple}
        \begin{split}
            &2\,\mathrm{Re}\bigl(\langle  x, \bm{A}x\rangle_{\mathcal{X}}\bigr)= -\bigl(\sigma_{\RNum{3}}\, |v_{\RNum{3}}(a)|^2+ \sigma_{\RNum{4}}\, |v_{\RNum{4}}(a)|^2\bigr).
        \end{split}
    \end{equation}
    }
    {It follows that $\bm{A}$ generates a contraction semigroup by Theorem~7.2.4 in \cite{jacob2012linear} and thus also that Assumption~\ref{assumption} holds with $\tilde{\bm{W}}_B$ representing \eqref{eqn: bound four string 1}, \eqref{eqn: bound four string 2} in terms of the boundary evaluations $\bm{H}(a)$ and $\bm{H}(b)$. 
    For $\delta \in \{a, b\}$ \[\|\bm{H}(\delta)x(\delta)\|^2 =\sum\limits_{j\in\{\RNum{1}, \RNum{2}, \RNum{3}, \RNum{4}\}} \left(|v_j(\delta)|^2+ |F_j(\delta)|^2\right)\]
    which implies that \eqref{eqn: birgit cond} cannot hold for $\delta=a$ and $\delta=b$. Indeed this follows directly from choosing any $x \in D(\bm{A})$ with
    $x(b) \neq 0$, $x_{\RNum{1}}(a) \neq 0$ and $x_{\RNum{3}}(a), x_{\RNum{4}}(a) = 0$.
    }
    Now fix $\gamma > 0$ such that 
$    \gamma \bm{H}(\zeta) \ge \pm \bm{P}_1^{-1}$ for all $\zeta \in [a, b]$.
    We can apply Proposition~\ref{lem: side ways} to all the equations
    \begin{equation}\label{eqn: all strings}
            \frac{\mathrm{d}}{\mathrm{d}t} x_j(t, \zeta) = \begin{pmatrix}
            0 & 1 \\
            1 & 0
        \end{pmatrix}\frac{\partial}{\partial \zeta} (\bm{H}_j(\zeta) x_j(t, \zeta))
    \end{equation}
    with $j = \RNum{1}, \RNum{2}, \RNum{3}, \RNum{4}.$
    {
    Note that
    \begin{align}
        \int\limits_{\sigma -2\gamma (b-a)}^{\tau + 2\gamma (b-a)} (x_{\RNum{2}}^*\bm{H}_{\RNum{2}}x_{\RNum{2}})(t, a) \, \mathrm{d}t \label{eqn: estimate middle} \le C_1 \int\limits_{S}^T \sum_{i=\RNum{3}}^{\RNum{4}}(x_{i}^*\bm{H}_{i}x_{i})(t, a)  \ \mathrm{d}t
    \end{align}
    }
    for some $C_1>0$ where $3\gamma(b-a) = \sigma$, $\tau > \sigma + 2\gamma (b-a)$ are fixed and $S =  \sigma - 3\gamma(b-a) = 0$, $T = \tau + 3\gamma (b-a)$.
     This follows since the boundary conditions imply that there is $\bm{R} \in \mathbb{R}^{2 \times 4}$ with $x_{\RNum{2}}(t, a) = \bm{R}\begin{pmatrix}
        x_{\RNum{3}}(b) \\
        x_{\RNum{4}}(b)
    \end{pmatrix}$ which implies that there is $C_2> 0$ with \begin{align*}
(x_{\RNum{2}}^*\bm{H}_{\RNum{2}}x_{\RNum{2}})(t, a) \le C_2\, ( (x_{\RNum{3}}^*\bm{H}_{\RNum{3}}x_{\RNum{3}} + x_{\RNum{4}}^*\bm{H}_{\RNum{4}}x_{\RNum{4}}))(t, b).
    \end{align*} 
    We then combine this with the application of Proposition~\ref{lem: side ways} {as in \eqref{eqn: side right left}} to \eqref{eqn: all strings} with $j = \RNum{3}, \RNum{4}$
    {
    which implies
    \begin{align*}
            \int\limits_{\sigma - 2\gamma(b-a)}^{\tau + 2\gamma (b-a)} (x_j^*\bm{H}_jx_j)(t, b) \, \mathrm{d}t &\le \mathrm{e}^{\kappa (b-a)}\int\limits_{S}^{T} (x_j^*\bm{H}_jx_j)(t, a)\, \mathrm{d}t
    \end{align*} for $j = \RNum{3}, \RNum{4}$ and a constant $\kappa > 0$
    }
    to get \eqref{eqn: estimate middle}. {Similarly, using Proposition~\ref{lem: side ways} applied to \eqref{eqn: all strings} with $j = \RNum{1}, \RNum{2}$ implies that  }
    \begin{equation}\label{eqn: trick again}
          \!\! \!\! \int\limits_\sigma^\tau (x_{\RNum{1}}^*\bm{H}_{\RNum{1}}x_{\RNum{1}})(t, a) \, \mathrm{d} t  \le {C_3} \!\!\int\limits_{\sigma - 2\gamma (b-a)}^{\tau + 2\gamma(b-a)} (x_{\RNum{2}}^*\bm{H}_{\RNum{2}}x_{\RNum{2}})(t, a) \, \mathrm{d} t,
    \end{equation}
    with ${C_3}>0$. Now, \eqref{eqn: four string simple} implies that there is $k' > 0$ such that ${\mathrm{Re}\, \langle \bm{A}x, x \rangle_{\mathcal{X}}} \le -k' \sum_{i=\RNum{3}}^{\RNum{4}}(x^*_{i} \bm{H}_{i} x_{i})(t, a).$
    We can combine this, \eqref{eqn: trick again} and \eqref{eqn: estimate middle} to obtain {that there is $k>0$ with}
    \begin{align*}
&         \|x(T)\|_{\mathcal{X}}^2 - \|x_0\|_{\mathcal{X}}^2  = {2}\int\limits_0^T\,\mathrm{Re}\,\langle  x(t), \bm{A}x(t)\rangle_{\mathcal{X}}  \, \mathrm{d}t \\
       &  \le {\frac{-2k'}{3}} \int\limits_0^T \sum_{i=\RNum{3}}^{\RNum{4}}(x^*_{i} \bm{H}_{i} x_{i})(t, a) \, \mathrm{d}t  \\ &
         {-\frac{2k'}{3C_1}}  \int\limits_{\sigma -2\gamma (b-a)}^{\tau + 2\gamma (b-a)} (x_{\RNum{2}}^*\bm{H}_{\RNum{2}}x_{\RNum{2}})(t, a) \, \mathrm{d}t \\ &- {\frac{2k'}{3C_1 {C_3}}}  \int\limits_\sigma^\tau (x_{\RNum{1}}^*\bm{H}_{\RNum{1}}x_{\RNum{1}})(t, a) \, \mathrm{d} t \le -k \int\limits_\sigma^\tau (x^*\bm{H}x)(t, a) \ \mathrm{d}t.
    \end{align*}
    Hence $(\bm{T}(t))_{t \ge 0}$ is exponentially stable by Theorem~\ref{stability basic}.

\section{Conclusion}
\label{sec:conclusion}
In this work, we used energy methods to analyze the long- and short-time behavior of distributed  pH systems. We have in particular given a characterization of exponential stability for such systems that generalizes a well-known sufficient condition in the literature. Using this new characterization, we can verify exponential stability for an example where the old condition is not applicable. In future work, we will use our condition to obtain sufficient conditions for the stability analysis of more complicated networks of pH systems.

\noindent
\textbf{Acknowledgement:} This work was funded by the Deutsche Forschungsgemeinschaft (DFG, German
Research Foundation) – Project-ID 531152215 – CRC 1701.Marco Roschkowski and Hannes Gernandt are with the Bergische Universität Wuppertal, Gaußstraße 20 (e-mail: \{gernandt,roschkowski\}@uni-wuppertal.de).  HG is also with the Fraunhofer Research Institution for Energy Infrastructures and Geotechnologies IEG, 03046 Cottbus, Germany.
This manuscript contains the same mathematical content as the article published in
\emph{IEEE Transactions on Automatic Control}
\cite{roschkowski2026two}.

\bibliographystyle{plain}
\bibliography{references}

\end{document}